\documentclass{eptcs}

\usepackage{amsmath,amssymb,amsthm}
\usepackage{enumerate,stmaryrd,mathpartir}
\usepackage{graphicx,graphbox,hyperref}
\usepackage{tikz-cd}

\theoremstyle{definition}
\newtheorem{definition}{Definition}
\newtheorem{theorem}{Theorem}

\newtheorem{observation}{Observation}

\newtheorem{lemma}{Lemma}
\newtheorem{remark}{Remark}

\newcommand{\A}{\mathbb{A}}

\newcommand{\D}{\mathbb{D}}

\newcommand{\bh}{\mathbf{H}}
\newcommand{\bv}{\mathbf{V}}

\newcommand{\corner}[1]{{\text{}^\ulcorner_\llcorner\!{#1}\!_\lrcorner^\urcorner}}

\newcommand{\cells}[4]{({\scriptstyle #1} {{\scriptstyle #2} \atop {\scriptstyle #3}} {\scriptstyle #4})}
\newcommand{\floatcells}[4]{\left({\scriptstyle #1} {{\scriptstyle #2} \atop {\scriptstyle #3}} {\scriptstyle #4}\right)}

\newcommand{\alice}{\texttt{Alice}}
\newcommand{\bob}{\texttt{Bob}}
 
\makeatletter
\providecommand{\leftsquigarrow}{%
  \mathrel{\mathpalette\reflect@squig\relax}%
}
\newcommand{\reflect@squig}[2]{%
  \reflectbox{$\m@th#1\rightsquigarrow$}%
}
\makeatother

\title{Cornering Optics}
\author{
Guillaume Boisseau
\institute{University of Oxford}
\thanks{This research is supported by the EPSRC.}
\and
Chad Nester
\institute{Tallinn University of Technology}
\thanks{This research was supported by the ESF funded Estonian IT Academy research measure (project 2014-2020.4.05.19-0001).}
\and
Mario Rom\'{a}n
\institute{Tallinn University of Technology \textsuperscript{\footnotesize{$\dagger$}}}
}

\def\titlerunning{Cornering Optics}
\def\authorrunning{G. Boisseau, C. Nester \& M. Rom\'{a}n}

\begin{document}

\maketitle

\begin{abstract}
We show that the category of optics in a monoidal category arises naturally from the free cornering of that category. Further, we show that the free cornering of a monoidal category is a natural setting in which to work with comb diagrams over that category. The free cornering admits an intuitive graphical calculus, which in light of our work may be used to reason about optics and comb diagrams. 
\end{abstract}


\section*{Introduction}

Optics in a monoidal category are a notion of bidirectional transformation, and have been something of a hot topic in recent years. In particular \emph{lenses}, which are optics in a cartesian monoidal category, play an important role in the theory of open games~\cite{Ghani2018}, compositional machine learning~\cite{Fong2019Backprop}, dialectica categories~\cite{Paiva1989}, functional programming~\cite{Pickering2017, Clarke2020}, the theory of polynomial functors ~\cite{Spivak2022}, and of course in the study of bidirectional transformations~\cite{Oles1982,Foster2005}. 

We recall the elementary presentation of the category $\mathsf{Optic}_\A$ of optics in a monoidal category $\A$. Objects $(A,B)$ are pairs of objects of $\A$. Arrows $\langle \alpha \mid \beta \rangle_M : (A,B) \to (C,D)$ consist of arrows $\alpha : A \to M \otimes C$ and $\beta : M \otimes D \to B$ of $\A$. It is helpful to visualize this as follows:
\[
\includegraphics[height=1.3cm,align=c]{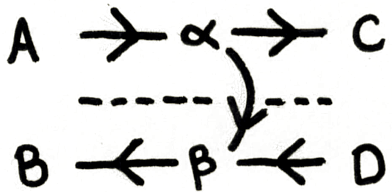}
\]
Arrows are subject to equations of the form $\langle \alpha(f \otimes 1_C) \mid \beta \rangle_N = \langle \alpha \mid (f \otimes 1_D)\beta \rangle_M$ for $f : M \to N$ in $\A$. This is often visualized as a sort of sliding between components, as in:
\[
\includegraphics[height=1.7cm,align=c]{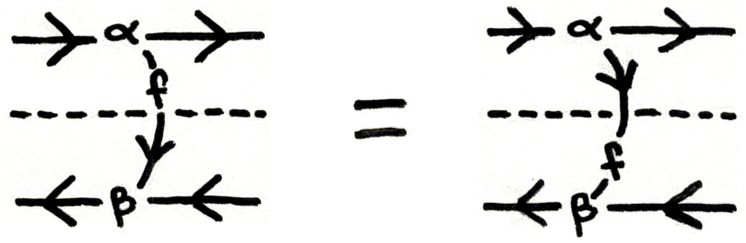}
\]
Equivalently, the hom-sets of $\mathsf{Optic}_\A$ can be given as a coend of hom-functors of $\A$:
\[
\mathsf{Optic}_\A((A,B),(C,D))
\cong
\int^M \A(A,M \otimes C) \times \A(M \otimes D,B)
\]
Composition is given by $\langle \alpha \mid \beta \rangle_M\langle \gamma \mid \delta \rangle_N = \langle \alpha(1_M \otimes \gamma) \mid (1_M \otimes \delta) \beta \rangle_{M \otimes N}$. Visually:
\[
\includegraphics[height=1.5cm,align=c]{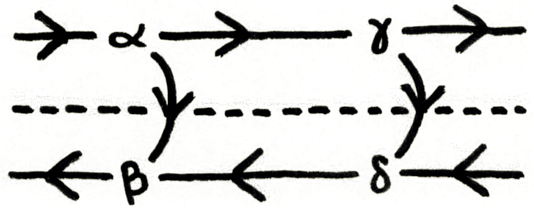}
\]
Identity arrows are given by $1_{(A,B)} = \langle 1_A \mid 1_B \rangle_I$.

Originally studied as an approach to concurrency by Nester~\cite{Nes21}, the \emph{free cornering} of a monoidal category is the double category obtained by freely adding companion and conjoint structure to it. The usual string diagrams for monoidal categories extend to an intuitive graphical calculus for the free cornering. The free cornering is the main piece of mathematical machinery in our development, and we give a detailed introduction to it in Section~\ref{sec:singledouble}.

Our main contribution is a characterisation of optics in a monoidal category in terms of its free cornering. More exactly, in Theorem~\ref{thm:fullsubcat} we show that the category of optics is a full subcategory of the horizontal cells of the free cornering. In addition to shedding some light on the nature of optics, this allows us to reason about them using the graphical calculus of the free cornering. We demonstrate this by using the graphical calculus to prove Lemmas~\ref{lem:inverselawful},~\ref{lem:lensdecompose},~\ref{lem:lawsimplylawful}, and~\ref{lem:lawfulimplieslaws}, which are a series of results originally due to Riley~\cite{Riley2018} concerning the lens laws. This occupies Section~\ref{sec:mainresults}.

Optics in a monoidal category can be seen as a special case of \emph{comb diagrams} in that category. Comb diagrams arose in the theory of quantum circuits~\cite{Chiribella2008}, and have since appeared in algebraic investigations of causal structure~\cite{Kissinger2017,Jacobs2021}. We suspect comb diagrams to be widely applicable, but there is not yet a commonly accepted algebra of comb diagrams. In Section~\ref{sec:combs} we give a notion of (single-sided) comb diagram in terms of the free cornering that coincides with the notion of comb diagram present in the work of Rom\'{a}n~\cite{Roman2021}. We demonstrate that the free cornering is a natural setting in which to work with comb diagrams, and consider this a further contribution of the present work. 

Our results are consequences of Lemma~\ref{lem:coendcorrespondence}, which characterises cells of the free cornering with a certain boundary shape in terms of coends. In particular, we make use of the soundness result for the graphical calculus of the free cornering due to Myers~\cite{Mye16}. The relevant definitions and the lemma itself are presented in Section~\ref{sec:alternationlemma}. The reader need not be familiar with coends to follow our development. While coends connect the free cornering to the wider literature through Lemma~\ref{lem:coendcorrespondence}, our work offers an alternate perspective that is conceptually simpler. 

In summary, we give a novel characterisation of optics and comb diagrams in a monoidal category in terms of the free cornering of that category. The graphical calculus of the free cornering allows one to work with these structures more easily. In addition to telling us something about the nature of optics and comb diagrams, our results suggest that the free cornering is worthy of further study in its own right. 

\section{Double Categories and the Free Cornering}\label{sec:singledouble}

In this section we set up the rest of our development by presenting the theory of single object double categories and the free cornering of a monoidal category. In this paper we consider only \emph{strict} monoidal categories, and in our development the term ``monoidal category'' should be read as ``strict monoidal category''. That said, we imagine that our results will hold in some form for arbitrary monoidal categories via the coherence theorem for monoidal categories~\cite{Mac71}. 

A \emph{single object double category} is a double category $\D$ with exactly one object. In this case $\D$ consists of a \emph{horizontal edge monoid} $\D_H = (\D_H, \otimes, I)$, a \emph{vertical edge monoid} $\D_V = (\D_V, \otimes, I)$, and a collection of \emph{cells}
\[
\includegraphics[height=1.7cm,align=c]{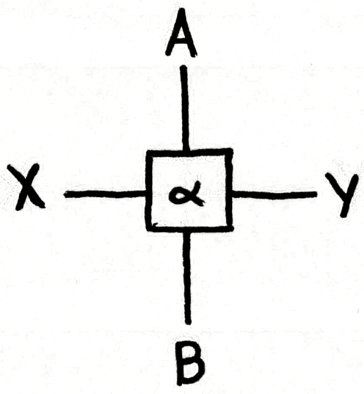}
\]
where $A,B \in \D_H$ and $X,Y \in \D_V$. We write $\D\cells{X}{A}{B}{Y}$ for the \emph{cell-set} of all such cells in $\D$. Given cells $\alpha,\beta$  where the right boundary of $\alpha$ matches the left boundary of $\beta$ we may form a cell $\alpha \vert \beta$ -- their \emph{horizontal composite} -- and similarly if the bottom boundary of $\alpha$ matches the top boundary of $\beta$ we may form $\frac{\alpha}{\beta}$ -- their \emph{vertical composite} -- with the boundaries of the composite cell formed from those of the component cells using $\otimes$. We depict horizontal and vertical composition, respectively, as in:
\begin{mathpar}
\includegraphics[height=1.7cm,align=c]{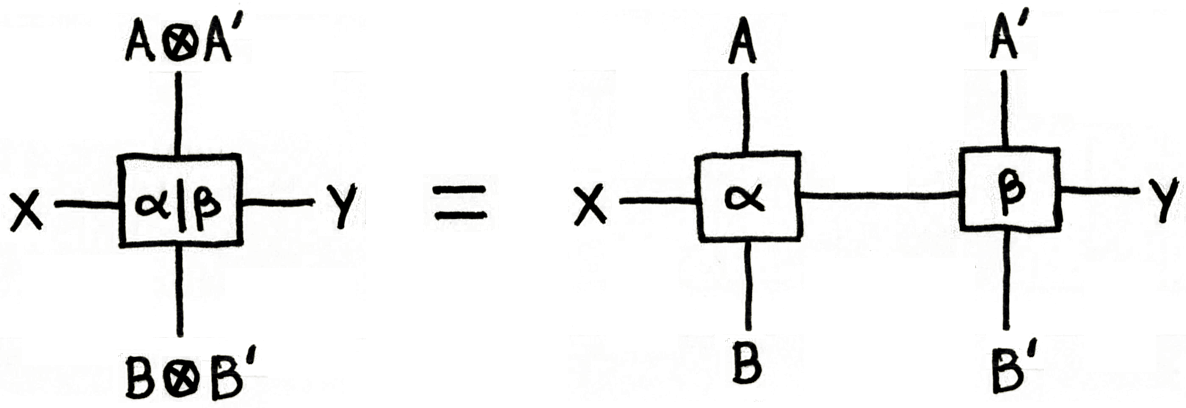}

\text{and}

\includegraphics[height=2.3cm,align=c]{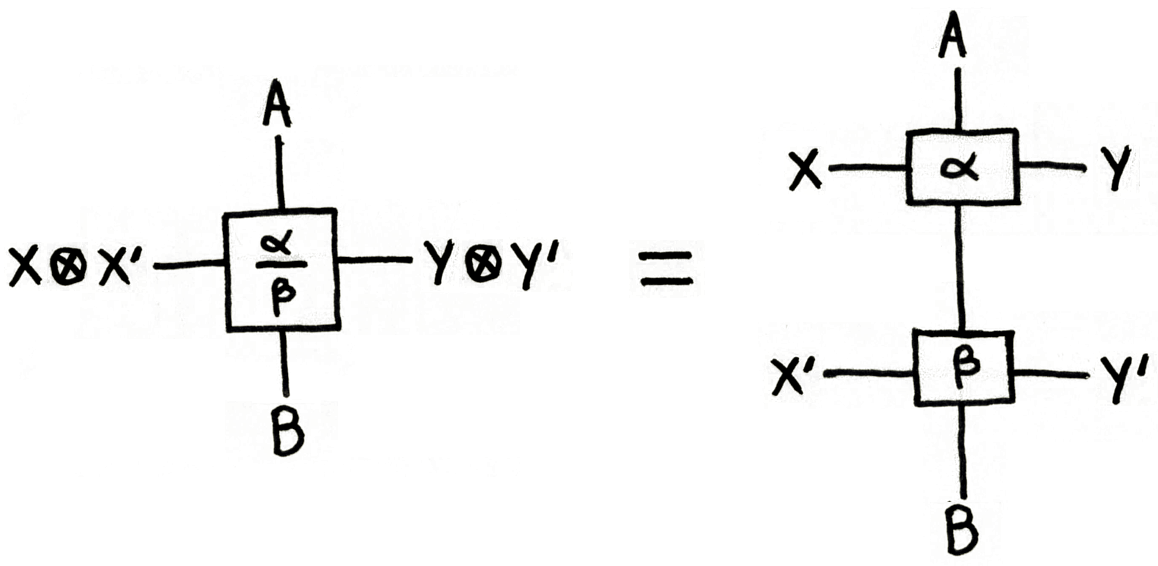}
\end{mathpar}
Horizontal and vertical composition of cells are required to be associative and unital. We omit wires of sort $I$ in our depictions of cells, allowing us to draw horizontal and vertical identity cells, respectively, as in:
\begin{mathpar}
\includegraphics[height=1.7cm,align=c]{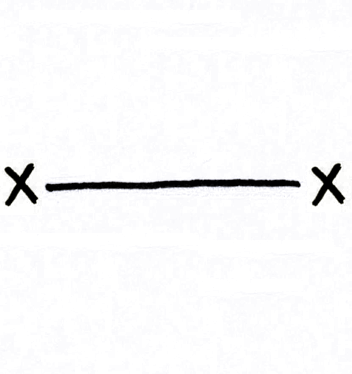}

\text{and}

\includegraphics[height=1.7cm,align=c]{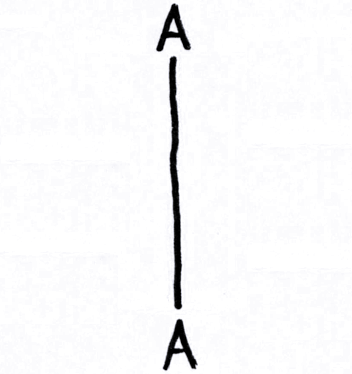}
\end{mathpar}
Finally, the horizontal and vertical identity cells of type $I$ must coincide -- we write this cell as $\square_I$ and depict it as empty space, see below on the left -- and vertical and horizontal composition must satisfy the interchange law. That is, $\frac{\alpha}{\beta}\vert \frac{\gamma}{\delta} = \frac{\alpha \vert \gamma}{\beta \vert \delta}$, allowing us to unambiguously interpret the diagram below on the right:
\begin{mathpar}
\includegraphics[height=1.7cm,align=c]{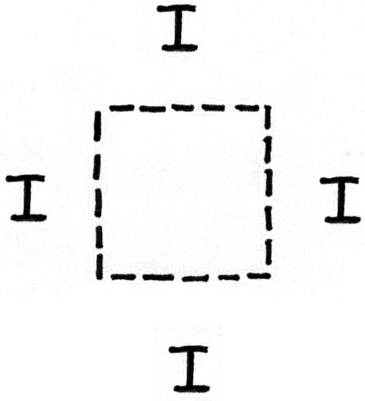}

\phantom{\text{and}}

\includegraphics[height=2.3cm,align=c]{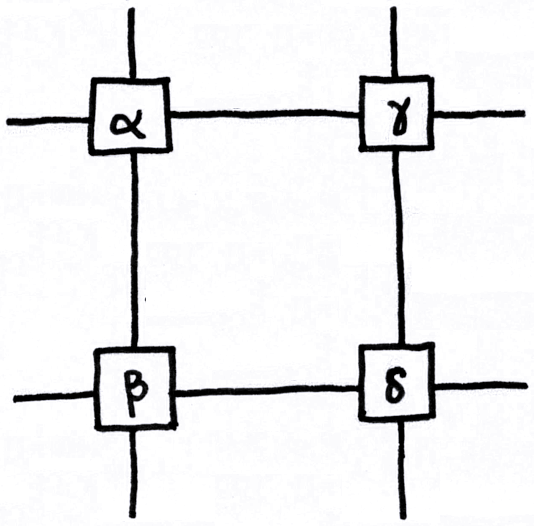}
\end{mathpar}

Every single object double category $\D$ defines strict monoidal categories $\bv \D$ and $\bh \D$, consisting of the cells for which the $\D_H$ and $\D_V$ valued boundaries respectively are all $I$, as in:
\begin{mathpar}
\includegraphics[height=1.7cm,align=c]{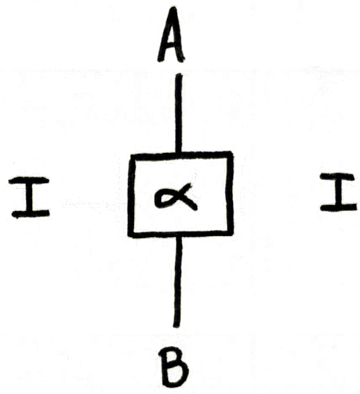}

\text{and}

\includegraphics[height=1.7cm,align=c]{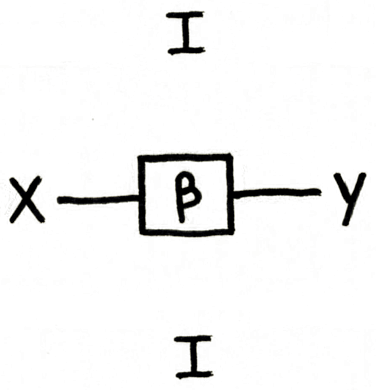}
\end{mathpar}
That is, the collection of objects of $\bv \D$ is $\D_H$, composition in $\bv \D$ is vertical composition of cells, and the tensor product in $\bv \D$ is given by horizontal composition:
\[
\includegraphics[height=1.7cm]{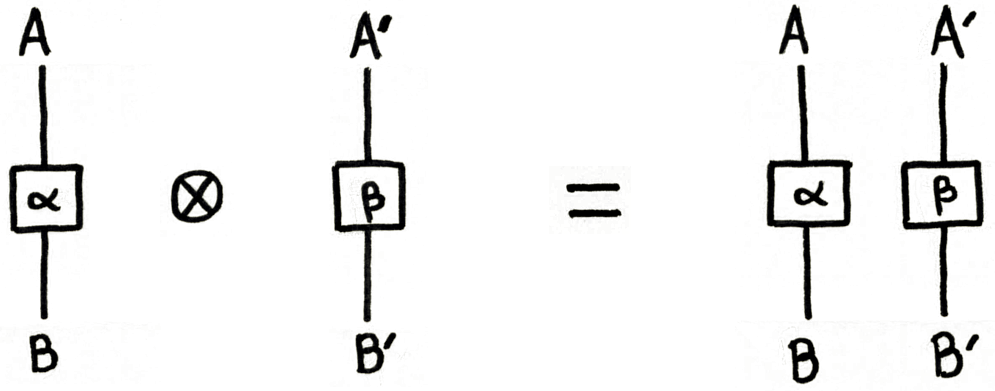}
\]
In this way, $\bv \D$ forms a strict monoidal category, which we call the category of $\emph{vertical cells}$ of $\D$. Similarly, $\bh \D$ is also a strict monoidal category (with collection of objects $\D_V$) which we call the \emph{horizontal cells} of $\D$.

Next, we introduce the free cornering of a monoidal category.

\begin{definition}[\cite{Nes21}]
  Let $\A$ be a monoidal category. We define the \emph{free cornering} of $\A$, written $\corner{\A}$, to be the free single object double category on the following data:
  \begin{itemize}
  \item The horizontal edge monoid $\corner{\A}_H = (\A_0, \otimes, I)$ is given by the objects of $\A$.
  \item The vertical edge monoid $\corner{\A}_V = (\A_0 \times \{ \circ,\bullet \})^*$ is the free monoid on the set $\A_0 \times \{ \circ, \bullet \}$ of polarized objects of $\A$ -- whose elements we write $A^\circ$ and $A^\bullet$.
  \item The generating cells consist of vertical cells $\corner{f}$ for each morphism $f : A \to B$ of $\A$ subject to equations as in:
  \[
  \includegraphics[height=1.7cm,align=c]{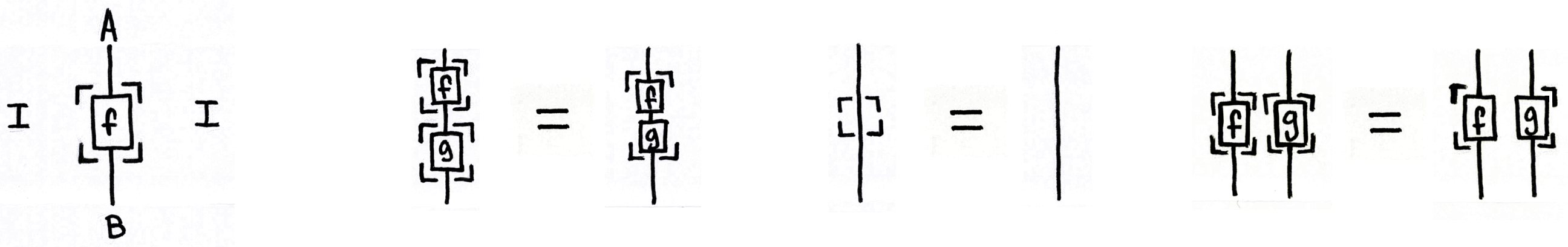}
  \]
  along with the following \emph{corner cells} for each object $A$ of $\A$:
  \[
  \includegraphics[height=1.7cm,align=c]{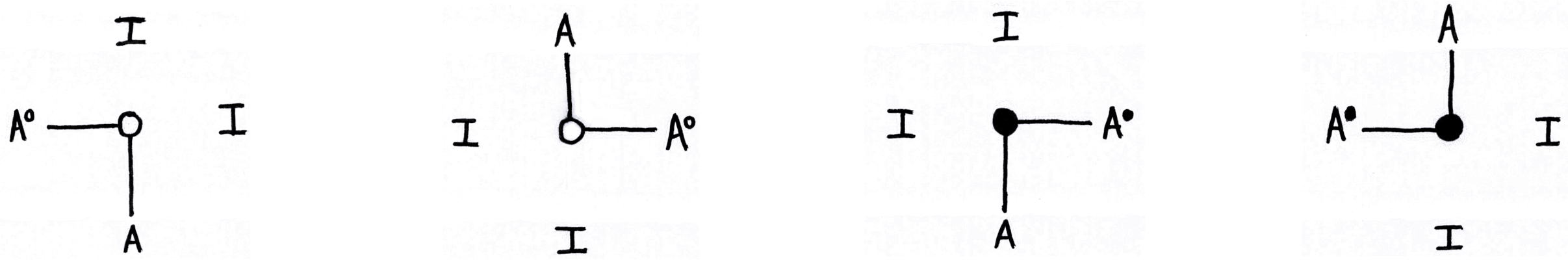}
  \]
  which are subject to the \emph{yanking equations}:
  \[
  \includegraphics[height=1cm,align=c]{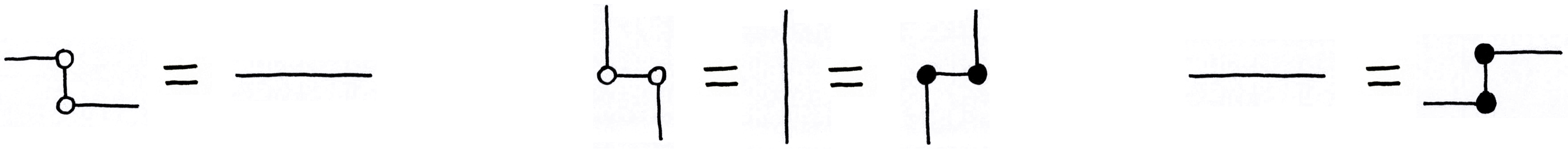}
  \]
  \end{itemize}
\end{definition}
For a precise development of free double categories see \cite{Fio08}. Briefly, cells are formed from the generating cells by horizontal and vertical composition, subject to the axioms of a double category in addition to any generating equations. The corner structure has been heavily studied under various names including \emph{proarrow equipment}, \emph{connection structure}, and \emph{companion and conjoint structure}. A good resource is the appendix of \cite{Shu08}.
 
We understand elements of $\corner{\A}_V$ as \emph{$\A$-valued exchanges}. Each exchange $X_1 \otimes \cdots \otimes X_n$ involves a left participant and a right participant giving each other resources in sequence, with $A^\circ$ indicating that the left participant should give the right participant an instance of $A$, and $A^\bullet$ indicating the opposite. For example say the left participant is $\alice$ and the right participant is $\bob$. Then we can picture the exchange $A^\circ \otimes B^\bullet \otimes C^\bullet$ as:
\[
\alice \rightsquigarrow
\includegraphics[height=1cm,align=c]{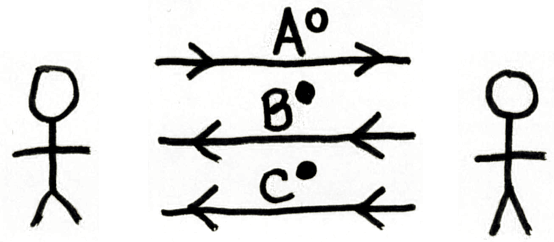}
\leftsquigarrow \bob
\]
Think of these exchanges as happening \emph{in order}. For example the exchange pictured above demands that first $\alice$ gives $\bob$ an instance of $A$, then $\bob$ gives $\alice$ an instance of $B$, and then finally $\bob$ gives $\alice$ an instance of $C$.

Cells of $\corner{\A}$ can be understood as \emph{interacting} morphisms of $\A$. Each cell is a method of obtaining the bottom boundary from the top boundary by participating in $\A$-valued exchanges along the left and right boundaries in addition to using the arrows of $\A$. For example, if the morphisms of $\A$ describe processes involved in baking bread, we might have the following cells of $\corner{\A}$:
\begin{mathpar}
\includegraphics[height=1.7cm,align=c]{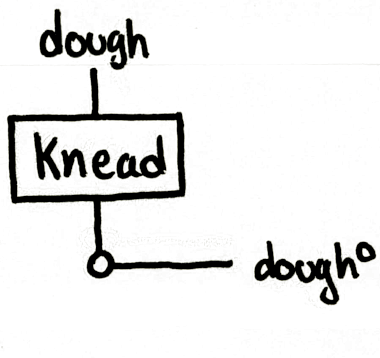}

\includegraphics[height=1.7cm,align=c]{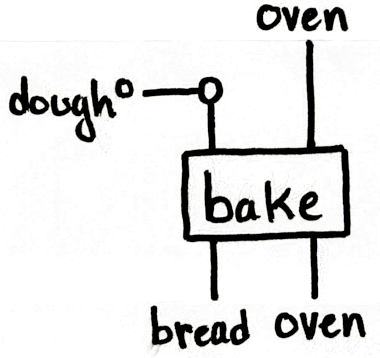}

\includegraphics[height=3cm,align=c]{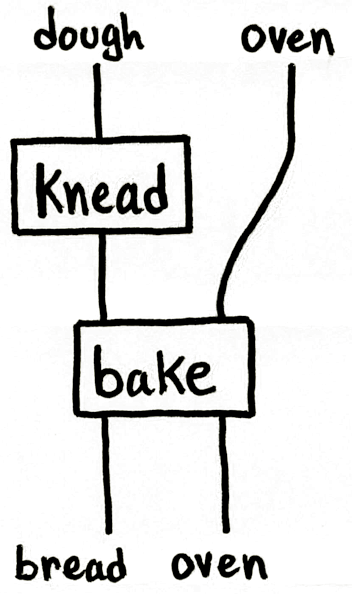}
\end{mathpar}
The cell on the left describes a procedure for transforming \texttt{dough} into nothing by \texttt{knead}ing it and sending the result away along the right boundary, and the cell in the middle describes a procedure for transforming an \texttt{oven} into \texttt{bread} and an \texttt{oven} by receiving \texttt{dough} along the left boundary and then using the \texttt{oven} to \texttt{bake} it. Composing these cells horizontally results in the cell on the right via the yanking equations. In this way the free cornering models concurrent interaction, with the corner cells capturing the flow of information across different components.

The vertical cells of the free cornering involve no exchanges, and as such are the cells of the original monoidal category:
\begin{lemma}[\cite{Nes21}]\label{lem:verticalcells}
There is an isomorphism of categories $\bv\,\corner{\A} \cong \A$. 
\end{lemma}

In comparison, the horizontal cells of the free cornering are not well understood.
In the sequel we will see that $\bh\,\corner{\A}$ contains $\mathsf{Optic}_\A$ as a full subcategory. 
\section{Alternation and Coends}\label{sec:alternationlemma}
In this section we prove a technical lemma characterizing certain cell-sets of $\corner{\A}$ as coends. 

\begin{definition}
  An element of $\corner{\A}_V$ is said to be \emph{$\bullet\circ$-alternating} in case it is of the form $A_1^\bullet \otimes B_1^\circ \otimes \cdots \otimes A_n^\bullet \otimes B_n^\circ$ for some $n \in \mathbb{N}$ such that $n > 0$. The \emph{alternation length} of a $\bullet\circ$-alternating element is defined to be the evident $n \in \mathbb{N}$. For example:
\begin{itemize}
\item $B^\bullet \otimes A^\circ$ is $\bullet\circ$-alternating with alternation length $1$. 
\item $A^\bullet \otimes B^\circ \otimes C^\bullet \otimes A^\circ$ is $\bullet\circ$-alternating with alternation length $2$.
\item $(A \otimes B)^\bullet \otimes I^\circ$ is $\bullet\circ$-alternating with alternation length $1$.
\item None of the following are $\bullet\circ$-alternating:
  \begin{mathpar}
    I
    
    A^\bullet \otimes B^\circ \otimes C^\circ
    
    A^\bullet \otimes B^\bullet
    
    A^\bullet
    
    (A \otimes B)^\circ \otimes B^\bullet
    
    A^\bullet \otimes B^\circ \otimes C^\bullet
    \end{mathpar}
\end{itemize}
\end{definition}

\begin{definition}
  A cell-set of the form $\corner{\A}\cells{I}{I}{I}{X}$ is said to be \emph{right-$\bullet\circ$-alternating} in case $X$ is $\bullet\circ$-alternating. The \emph{alternation depth} of a right-$\bullet\circ$-alternating cell-set is the alternation length of its right boundary. 
\end{definition}

\begin{lemma}\label{lem:coendcorrespondence}
  If $\corner{\A}\cells{I}{I}{I}{X}$ is right-$\bullet\circ$-alternating with alternation depth $n$ and $X = A_1^\bullet \otimes B_1^\circ \otimes \cdots \otimes A_n^\bullet \otimes B_n^\circ$ then
  \[
  \corner{\A}\floatcells{I}{I}{I}{X}
  \cong
  \int^{M_1,\ldots,M_{n-1}}\prod^n_{i=1}\A(M_{i-1} \otimes A_i, M_i \otimes B_i)
  \]
  where $M_0 = M_n = I$. 
\end{lemma}
\begin{proof} By inspecting the generating cells of $\corner{\A}$ and making use of Lemma~\ref{lem:verticalcells} we find that any cell of $\corner{\A}\cells{I}{I}{I}{X}$ is necessarily of the form:
  \[
  \includegraphics[height=5cm,align=c]{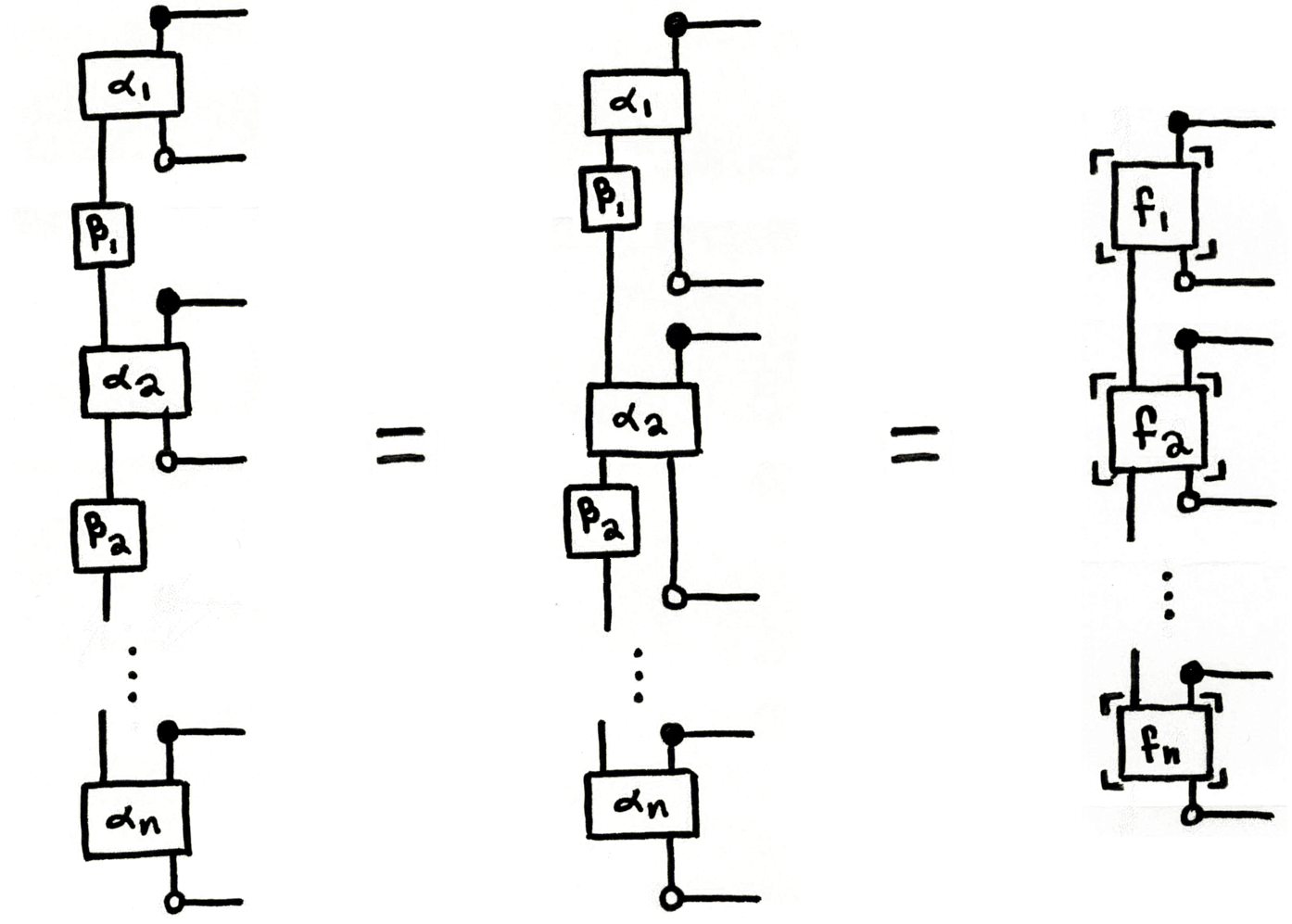}
  \]
  Thus cells of $\corner{\A}\cells{I}{I}{I}{X}$ may be written as $n$-tuples $\langle f_1 \mid \cdots \mid f_n \rangle$. As a consequence of Myers' soundness result for the graphical calculus~\cite{Mye16}, we know that two cells $\langle f_1 \mid \cdots \mid f_n \rangle$ and $\langle g_1 \mid \cdots \mid g_n \rangle$ of $\corner{\A}\cells{I}{I}{I}{X}$ are equal iff they are deformable into each other modulo the equations of $\A$. Consider that all \emph{local} deformations $\langle \cdots \mid f_i \mid f_{i+1} \mid \cdots \rangle = \langle \cdots \mid g_i \mid g_{i+1} \mid \cdots \rangle$ are of the form:
\[
\includegraphics[height=3.7cm,align=c]{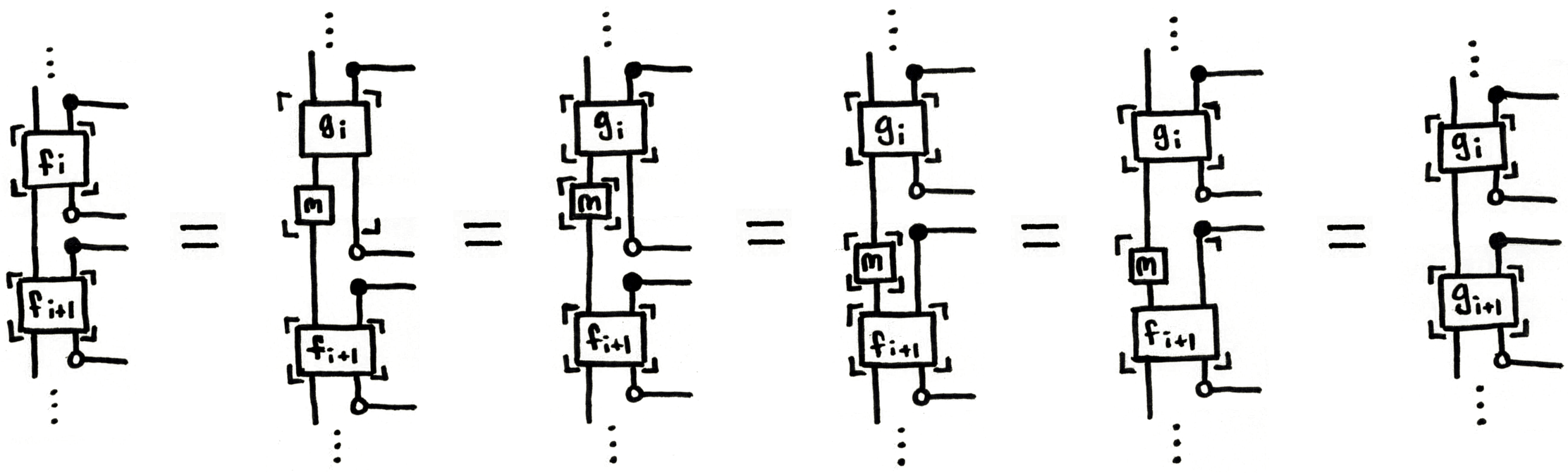}
\]
where $f_i = g_i(m \otimes 1)$ and $g_{i+1} = (m \otimes 1)f_{i + 1}$. Now, the only way $\langle f_1 \mid \cdots \mid f_n \rangle$ and $\langle g_1 \mid \cdots \mid g_n \rangle$ can be equal is by (repeated) parallel local deformation of the associated diagrams, as in:
\[
\includegraphics[height=5cm,align=c]{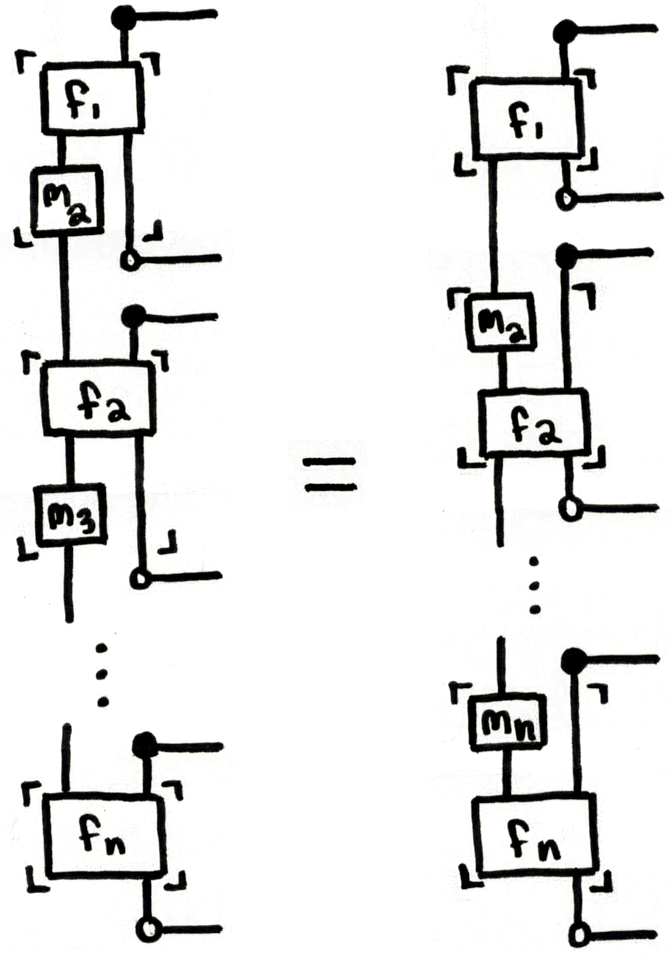}
\]
Thus, $\corner{\A}\cells{I}{I}{I}{X}$ is the set of (appropriately typed) $n$-tuples $\langle f_1 \mid \cdots \mid f_n \rangle$ of morphisms of $\A$, quotiented by equations of the form:
\[ \langle f_1(m_2 \otimes 1) \mid f_2(m_3 \otimes 1) \mid \cdots \mid f_n \rangle = \langle f_1 \mid (m_2 \otimes 1)f_2 \mid \cdots \mid (m_n \otimes 1)f_n \rangle \]
which is precisely to say that the claim holds.
\end{proof}

\begin{remark}\label{rmk:alternationdual}
  There is an obvious dual notion of \emph{left-$\circ\bullet$-alternating} cell-set for which a version of Lemma~\ref{lem:coendcorrespondence} holds.
\end{remark}

\section{Optics and the Free Cornering}\label{sec:mainresults}
In this section we use Lemma \ref{lem:coendcorrespondence} to show that $\mathsf{Optic}_\A$ is a full subcategory of $\bh\,\corner{\A}$ for any monoidal category $\A$. We then briefly discuss lenses, and illustrate the power of the graphical calculus for $\corner{\A}$ by reproving a correspondence between lenses satisfying the the lens laws and lenses that are comonoid homomorphisms with respect to a certain comonoid structure. These results about lenses are originally due to Riley~\cite{Riley2018}, and were also used to demonstrate Boisseau's approach to string diagrams for optics~\cite{Boisseau2020}. We end with Observation~\ref{obs:tele}, which discusses the relation of teleological categories~\cite{Hedges2017} to the free cornering. 

\begin{theorem}\label{thm:fullsubcat}
  Let $\A$ be a monoidal category. Then $\mathsf{Optic}_\A$ is the full subcategory of $\bh\,\corner{\A}$ on objects of the form $A^\circ \otimes B^\bullet$ for $A,B \in \A_0$. 
\end{theorem}
\begin{proof}
  We begin by noticing that
  \[
  \bh\,\corner{\A}(A^\circ \otimes B^\bullet,C^\circ \otimes D^\bullet)
  \cong
  \corner{\A}\cells{I}{I}{I}{A^\bullet \otimes C^\circ \otimes D^\bullet \otimes B^\circ}
  \]
  via:
  \[
  \includegraphics[height=1.2cm,align=c]{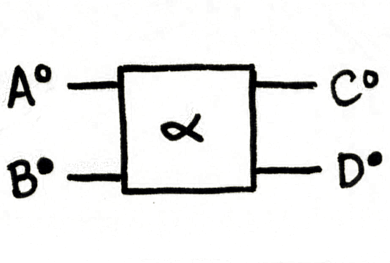}
  \hspace{0.3cm}
  \mapsto
  \hspace{0.3cm}
  \includegraphics[height=1.2cm,align=c]{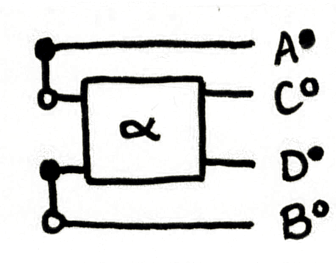}
  \hspace{1cm}
  \text{and}
  \hspace{1cm}
  \includegraphics[height=1.2cm,align=c]{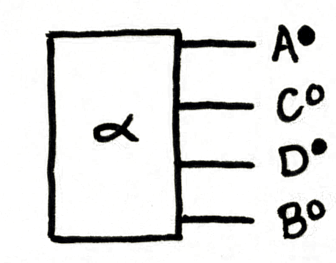}
  \hspace{0.3cm}
  \mapsto
  \hspace{0.3cm}
  \includegraphics[height=1.2cm,align=c]{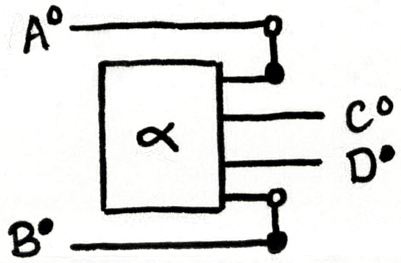}
  \]
  This cell-set is right-$\bullet\circ$-alternating of depth 2, and so we have:
  \[
  \corner{\A}\floatcells{I}{I}{I}{A^\bullet \otimes C^\circ \otimes D^\bullet \otimes B^\circ}
  \cong
  \int^{M \in \A} \A(A,M \otimes C) \times \A(M \otimes D,B) 
  \]
  Now we already know that
  \[
  \int^{M \in \A} \A(A,M \otimes C) \times \A(M \otimes D,B) \\
  \cong
  \mathsf{Optic}_\A((A,B),(C,D))
  \]
  and so we have a correspondence between arrows of $\bh\,\corner{\A}$ and arrows of $\mathsf{Optic}_\A$:
  \[
  \bh\,\corner{\A}(A^\circ \otimes B^\bullet,C^\circ \otimes D^\bullet)
  \cong
  \mathsf{Optic}_\A((A,B),(C,D))
  \]
  In particular, we know that arrows in $\bh\,\corner{\A}(A^\circ \otimes B^\bullet ,C^\circ \otimes D^\bullet)$ are equivalently optics $\langle \alpha \mid \beta \rangle_M$ as below left, and that the equations between optics -- below right -- capture all equations in $\bh\,\corner{\A}(A^\circ \otimes B^\bullet,C^\circ \otimes D^\bullet)$:
  \begin{mathpar}
  \includegraphics[height=2.5cm,align=c]{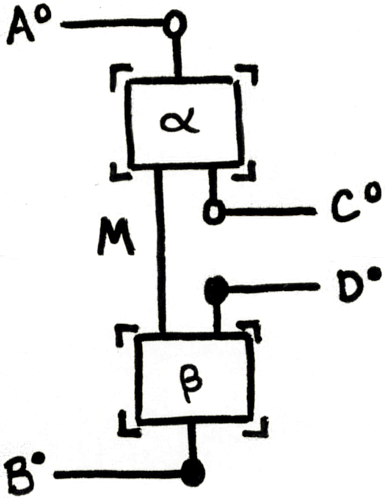}

  \includegraphics[height=2.5cm,align=c]{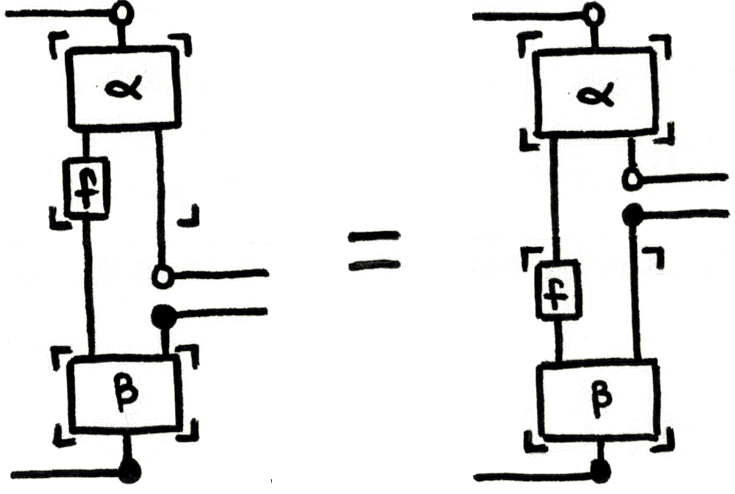}
  \end{mathpar}
  Next, given arrows $\langle \alpha \mid \beta \rangle_M : (A,B) \to (C,D)$ and $\langle \gamma \mid \delta \rangle_N : (C,D) \to (E,F)$ of $\mathsf{Optic}_\A$, we find that composing the corresponding arrows of $\bh\,\corner{\A}$ yields the arrow corresponding to $\langle \alpha(1_M \otimes \gamma) \mid (1_M \otimes \delta)\beta \rangle_{M \otimes N} = \langle \alpha \mid \beta \rangle_M \langle \gamma \mid \delta \rangle_N$ as in:
  \[
  \includegraphics[height=3.5cm,align=c]{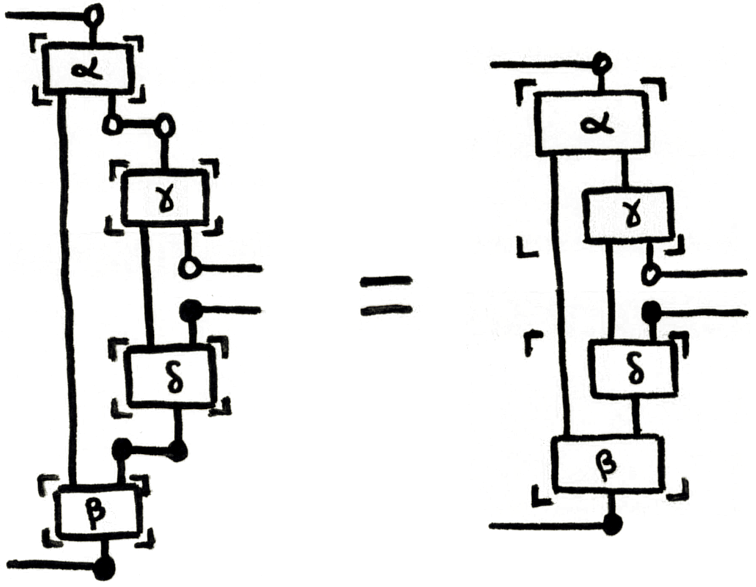}
  \]
  Further, the identity on $A^\circ \otimes B^\bullet$ in $\bh\,\corner{\A}$ corresponds to the $1_{(A,B)} = \langle 1_A \mid 1_B \rangle_I$ in $\mathsf{Optic}_\A$ as in:
  \[
  \includegraphics[height=1.5cm,align=c]{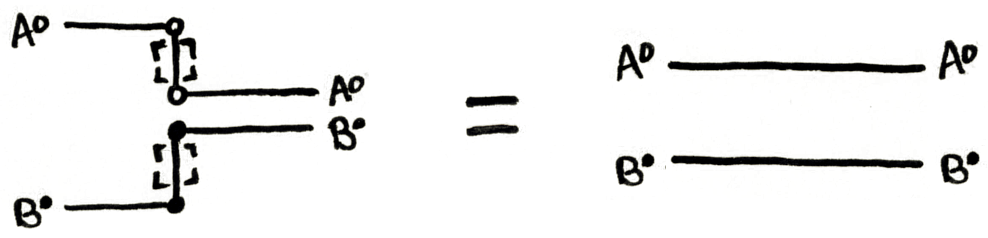}
  \]
  The result is thus proven.
\end{proof}

\begin{remark}\label{rmk:opoptics}
  Following Remark~\ref{rmk:alternationdual}, a similar argument gives that if $A$ is symmetric monoidal then $\bh\,\corner{\A}^{op}$ also contains $\mathsf{Optic}_\A$ as the full subcategory on those objects of the form $A^\bullet \otimes B^\circ$.
\end{remark}

\begin{remark}
  If $\A$ is a \emph{symmetric} monoidal category then $\mathsf{Optic}_\A$ is itself monoidal~\cite{Riley2018}. We remark that while $\mathsf{Optic}_\A$ remains a subcategory of $\bh\,\corner{\A}$ in this case, it is not a \emph{monoidal} subcategory. That is, the tensor product of optics is \emph{not} given by the tensor product in $\bh\,\corner{\A}$. 
\end{remark}

As an illustration of our approach, we consider the characterisation of the lens laws given in \cite{Riley2018}. Say that an optic is \emph{homogeneous} in case it is contained in the full subcategory of $\mathsf{Optic}_\A$ on objects $(A,A)$ for some $A \in \A_0$. Notice that every object of this subcategory is a comonoid in $\bh\,\corner{\A}$ , with the comultiplication and counit given as in:
\begin{mathpar}
\includegraphics[height=1.2cm,align=c]{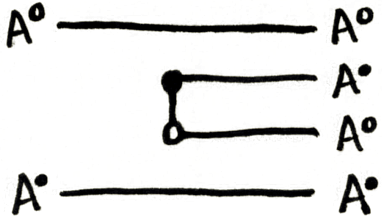}

\includegraphics[height=1.2cm,align=c]{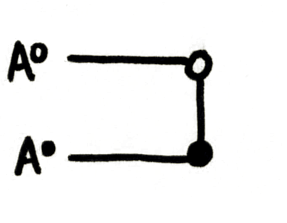}
\end{mathpar}
where the comonoid axioms hold as in:
\begin{mathpar}
\includegraphics[height=1.2cm,align=c]{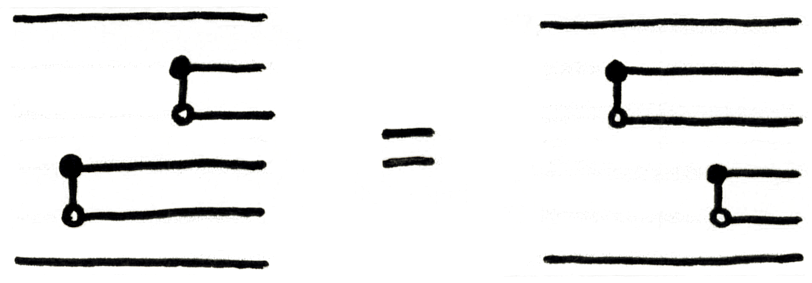}

\includegraphics[height=1.2cm,align=c]{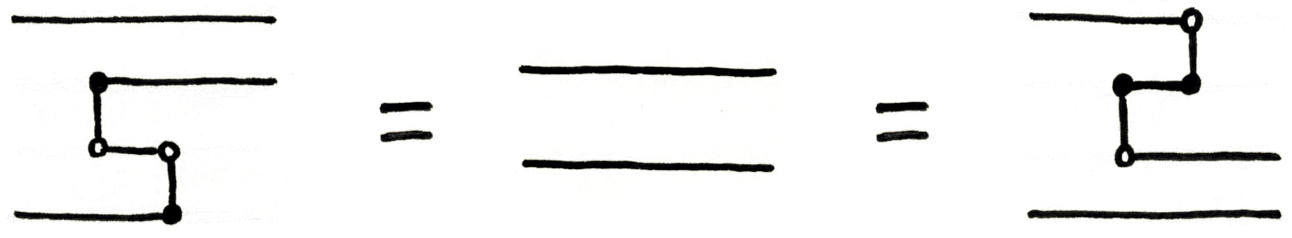}
\end{mathpar}

\begin{definition}[\cite{Riley2018}]
  A homogeneous optic $h : (A,A) \to (B,B)$ of $\mathsf{Optic}_\A$ is called \emph{lawful} in case the following equations hold in $\bh\,\corner{\A}$:
  \begin{mathpar}
    \includegraphics[height=1.2cm,align=c]{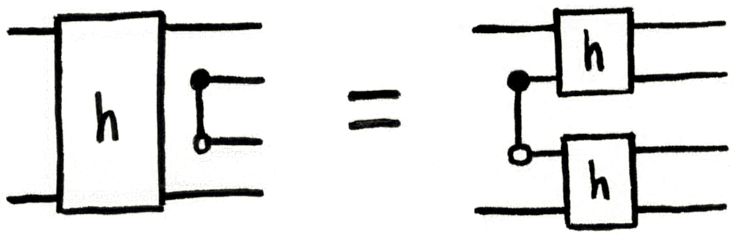}
    
    \includegraphics[height=1.2cm,align=c]{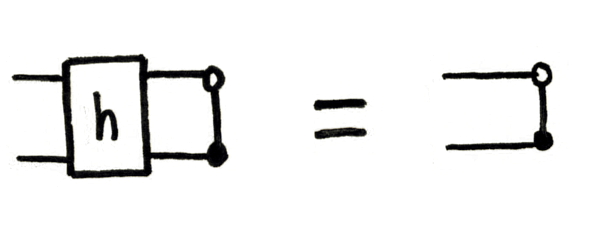}
  \end{mathpar}
  That is, in case $h$ is a comonoid homomorphism with respect to the comonoid structure given above. 
\end{definition}

\begin{lemma}[\cite{Riley2018}]\label{lem:inverselawful}
  If $h = \langle \alpha \mid \beta \rangle_M : (A,A) \to (B,B)$ in $\mathsf{Optic}_\A$ with $\alpha$ and $\beta$ mutually inverse, then $h$ is lawful. 
\end{lemma}
\begin{proof}
  \begin{mathpar}
    \includegraphics[height=3.5cm,align=c]{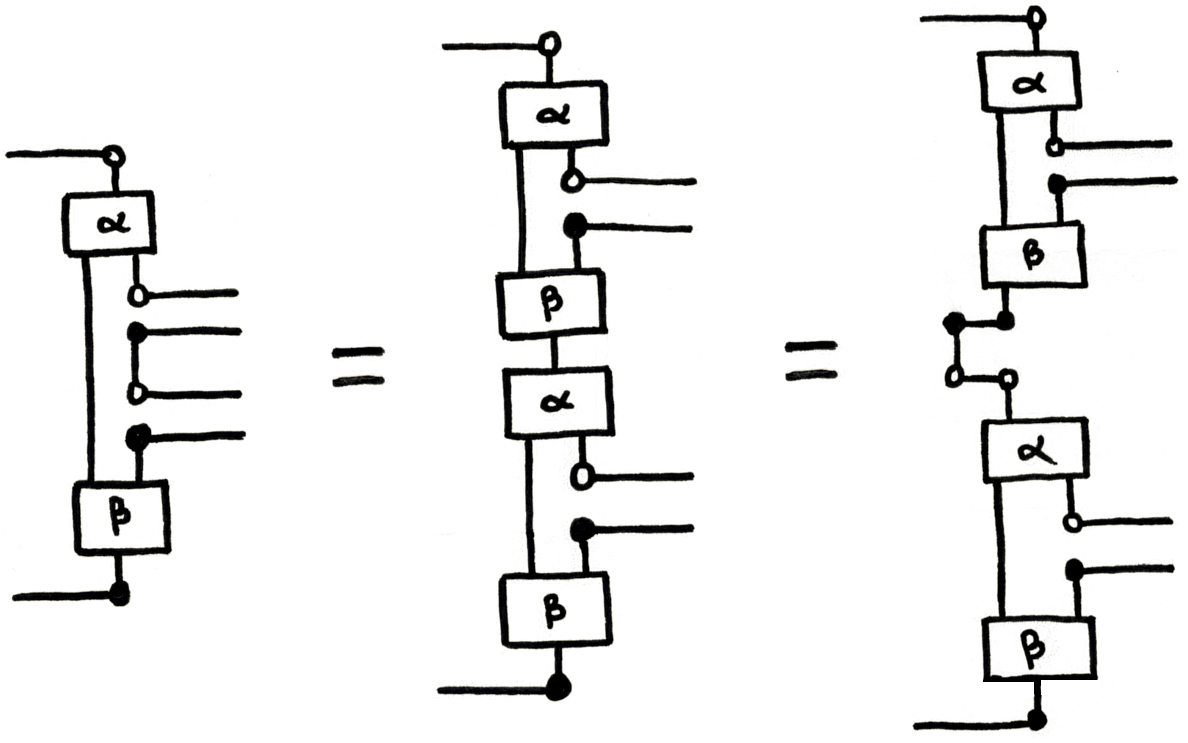}

    \text{and}
    
    \includegraphics[height=2cm,align=c]{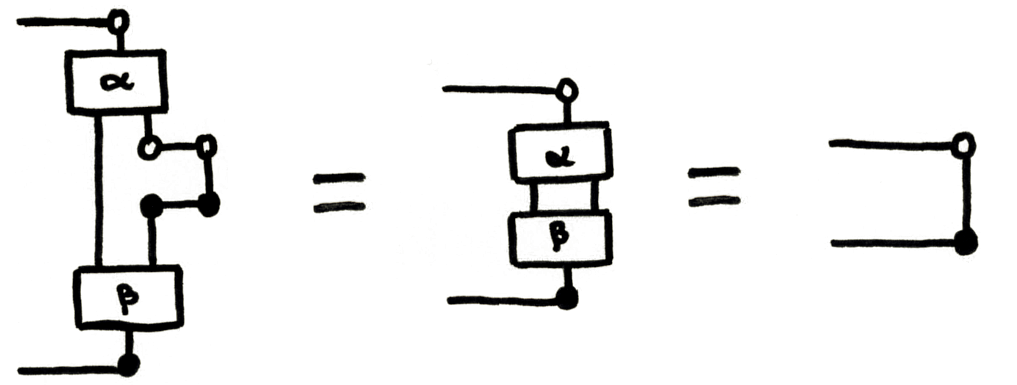}    
  \end{mathpar}
\end{proof}

Recalling the algebraic characterisation of cartesian monoidal categories~\cite{Fox76}, we denote the commutative comonoid structure in a cartesian monoidal category as follows:
\begin{mathpar}
  \includegraphics[height=1.7cm,align=c]{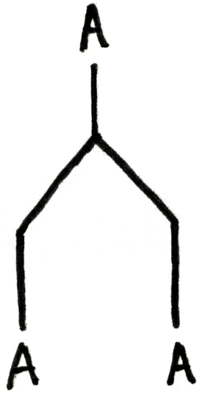}

  \includegraphics[height=1.7cm,align=c]{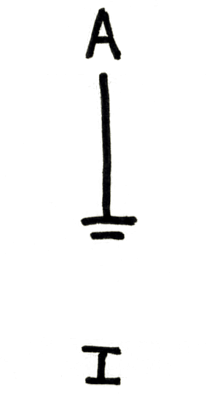}
\end{mathpar}
This structure must satisfy the commutative comonoid axioms:
\begin{mathpar}
  \includegraphics[height=1.2cm,align=c]{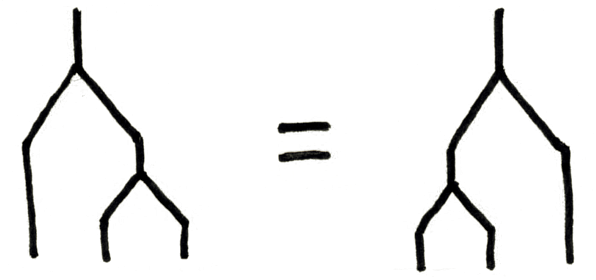}

  \includegraphics[height=1.2cm,align=c]{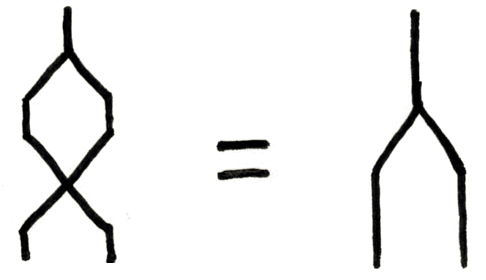}

  \includegraphics[height=1.2cm,align=c]{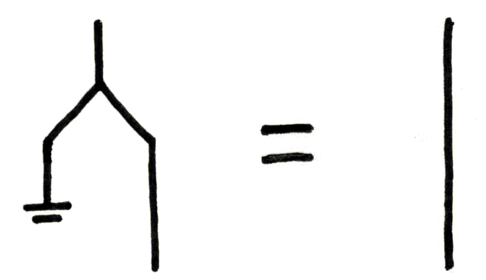}
\end{mathpar}
Must further be coherent with respect to the monoidal structure:
\begin{mathpar}
  \includegraphics[height=1.7cm,align=c]{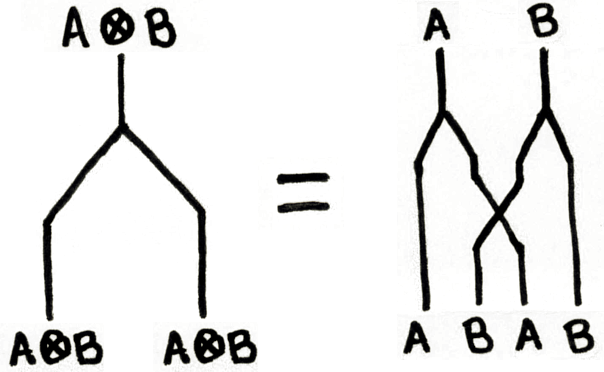}

  \includegraphics[height=1.7cm,align=c]{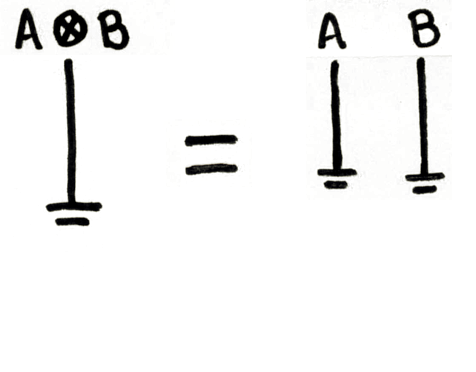}
\end{mathpar}
And every morphism $f$ of the category in question must be a comonoid homomorphism:
\begin{mathpar}
  \includegraphics[height=1.2cm,align=c]{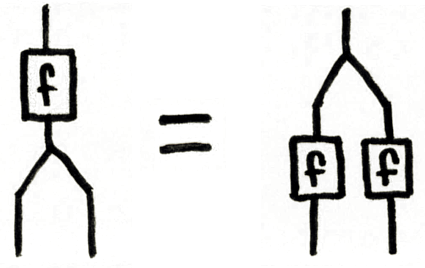}

  \includegraphics[height=1.2cm,align=c]{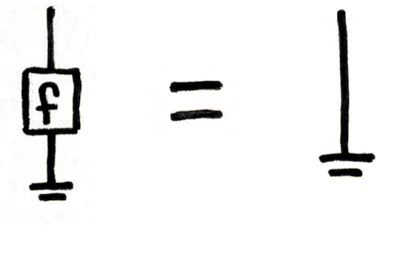}
\end{mathpar}

\begin{lemma}[\cite{Riley2018}]\label{lem:lensdecompose}
  Let $\A$ be a cartesian monoidal category, and let $h = \langle \alpha \mid \beta \rangle_M : (A,A) \to (B,B)$ be a homogeneous optic in $\A$. Then there exist arrows $\mathsf{get} : A \to B$ and $\mathsf{put} : A \otimes B \to A$ of $\A$ such that:
  \[
  \includegraphics[height=2cm,align=c]{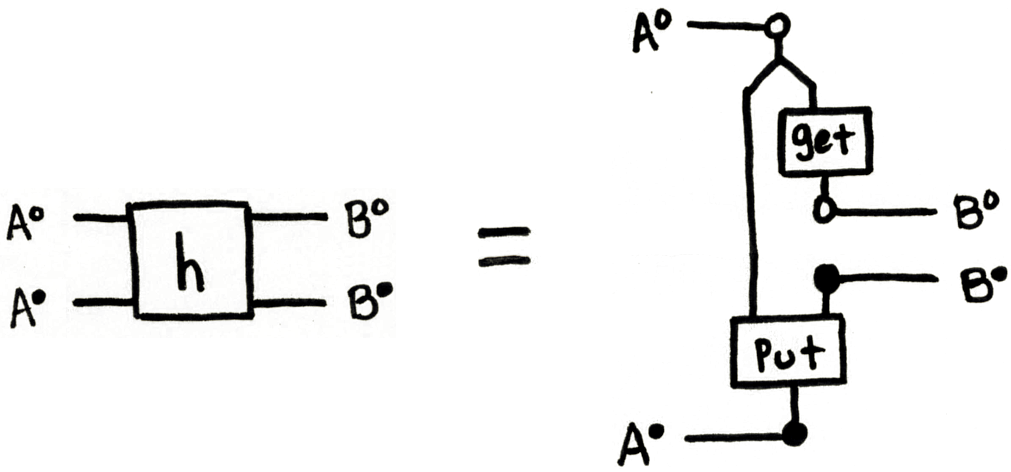}
  \]
\end{lemma}
\begin{proof}
  We have:
  \[
  \includegraphics[height=2.5cm,align=c]{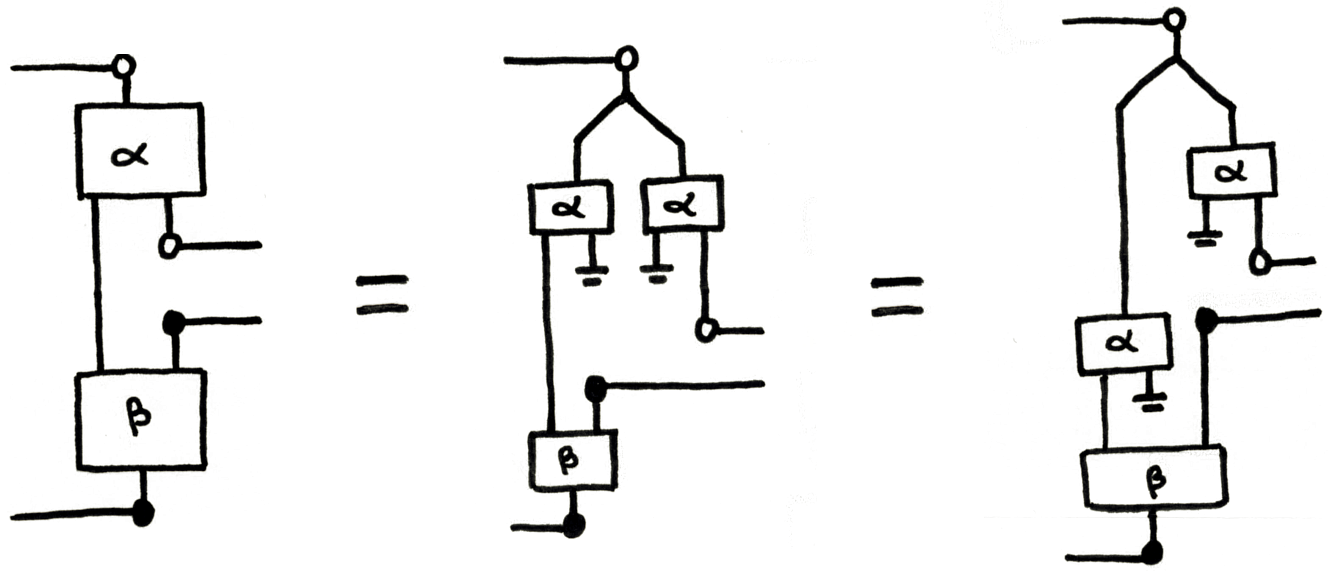}
  \]
  and so the claim follows via:
  \begin{mathpar}
    \includegraphics[height=1.7cm,align=c]{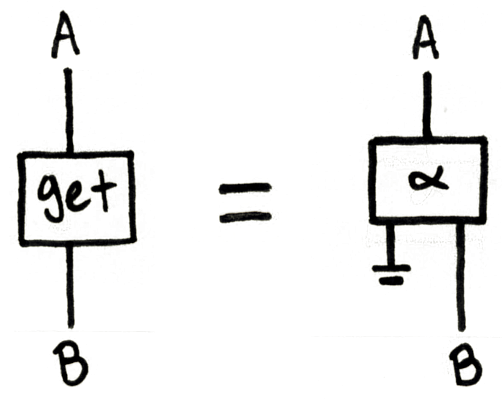}
    
    \includegraphics[height=1.7cm,align=c]{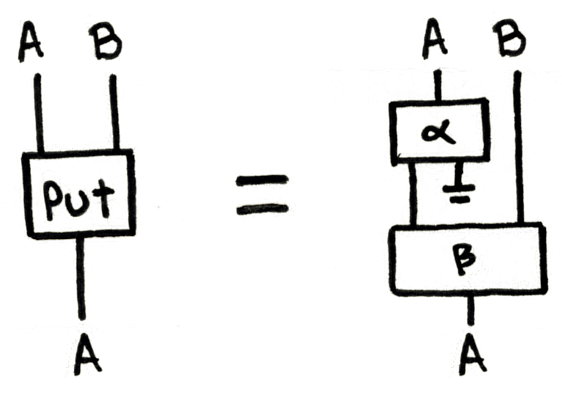}
  \end{mathpar}
\end{proof}

Homogeneous optics in cartesian monoidal categories are called \emph{lenses}. We write $[\mathsf{put} \mid \mathsf{get}] : (A,A) \to (B,B)$ for the lens specified by appropriate $\mathsf{put}$ and $\mathsf{get}$ arrows in the above manner. 

\begin{definition}[\cite{Foster2005}]
  A lens $[\mathsf{put} \mid \mathsf{get}] : (A,A) \to (B,B)$ is is said satisfy the \emph{lens laws} in case:
  \begin{mathpar}
    \includegraphics[height=2cm,align=c]{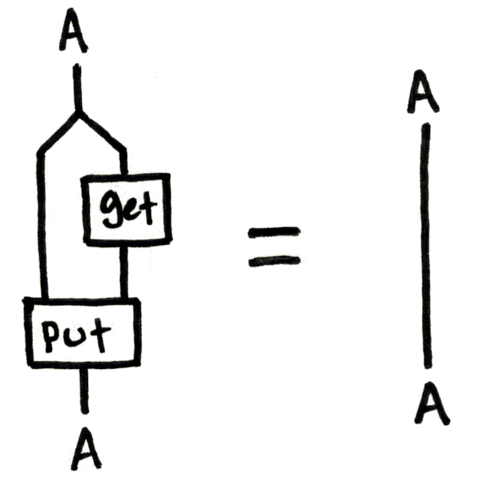}
    
    \includegraphics[height=1.7cm,align=c]{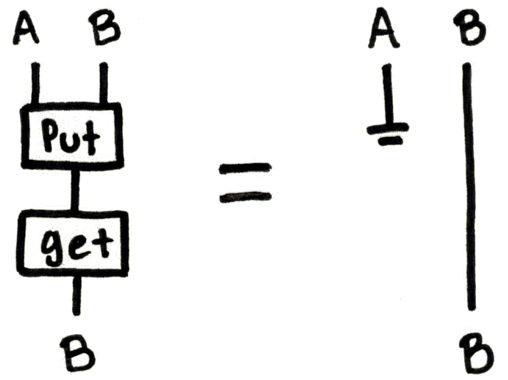}
    
    \includegraphics[height=1.7cm,align=c]{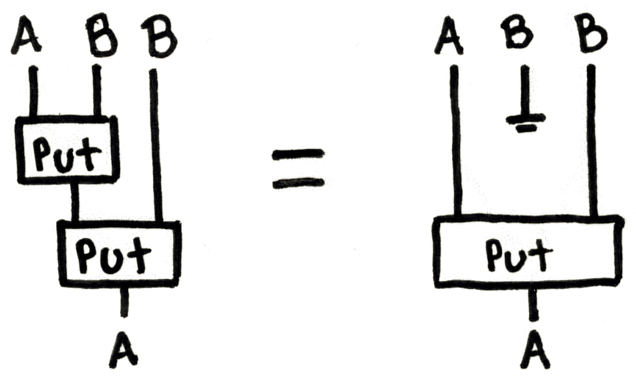}
  \end{mathpar}
\end{definition}  

\begin{lemma}[\cite{Riley2018}]\label{lem:lawsimplylawful}
  If a lens $h = [ \mathsf{put} \mid \mathsf{get} ] : (A,A) \to (B,B)$ satisfies the lens laws then it is lawful. 
\end{lemma}
\begin{proof}
  For the counit we have:
  \[
  \includegraphics[height=1.7cm,align=c]{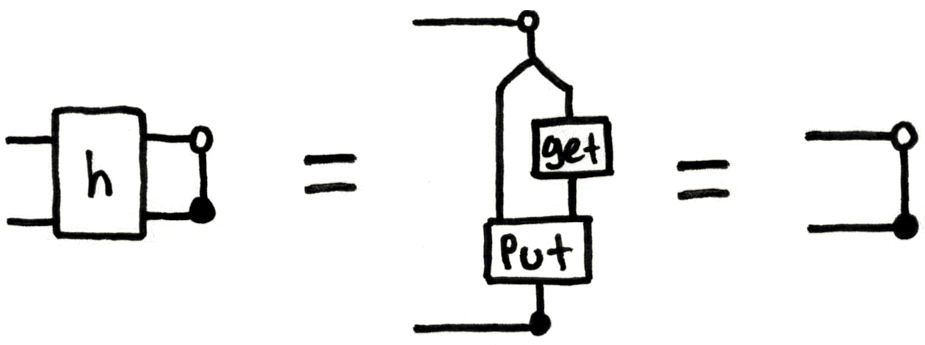}
  \]
  And for the comultiplication:
  \[
  \includegraphics[height=4cm,align=c]{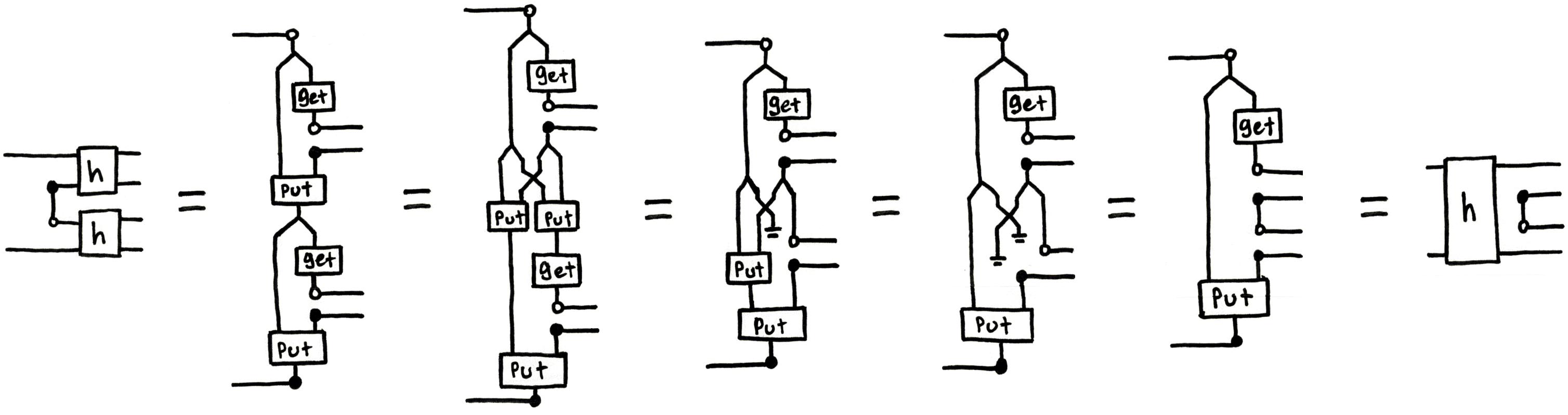}
  \]
\end{proof}

\begin{lemma}[\cite{Riley2018}]\label{lem:lawfulimplieslaws}
  If a lens $h = [\mathsf{put} \mid \mathsf{get} ] : (A,A) \to (B,B)$ is lawful and $B$ is inhabited in the sense that there is an arrow $k : 1 \to B$ in $\A$, then it satisfies the lens laws.
\end{lemma}
\begin{proof}
  The first lens law holds as in:
  \[
  \includegraphics[height=2.5cm,align=c]{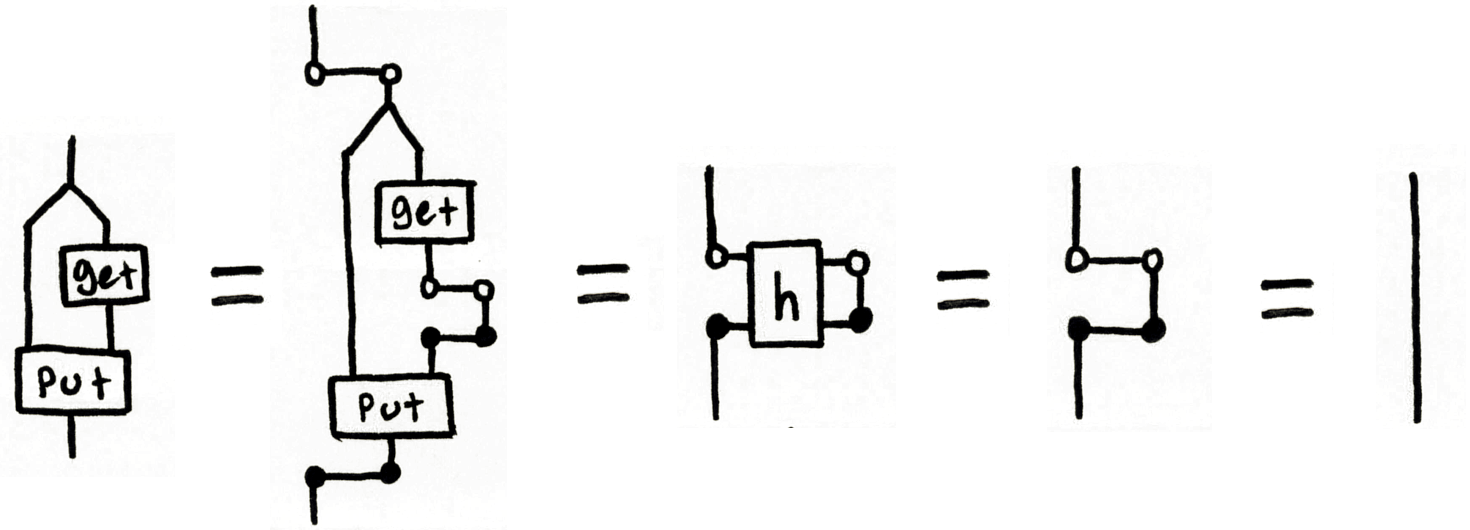}
  \]
  The second lens law holds as in:
  \[
  \includegraphics[height=4cm,align=c]{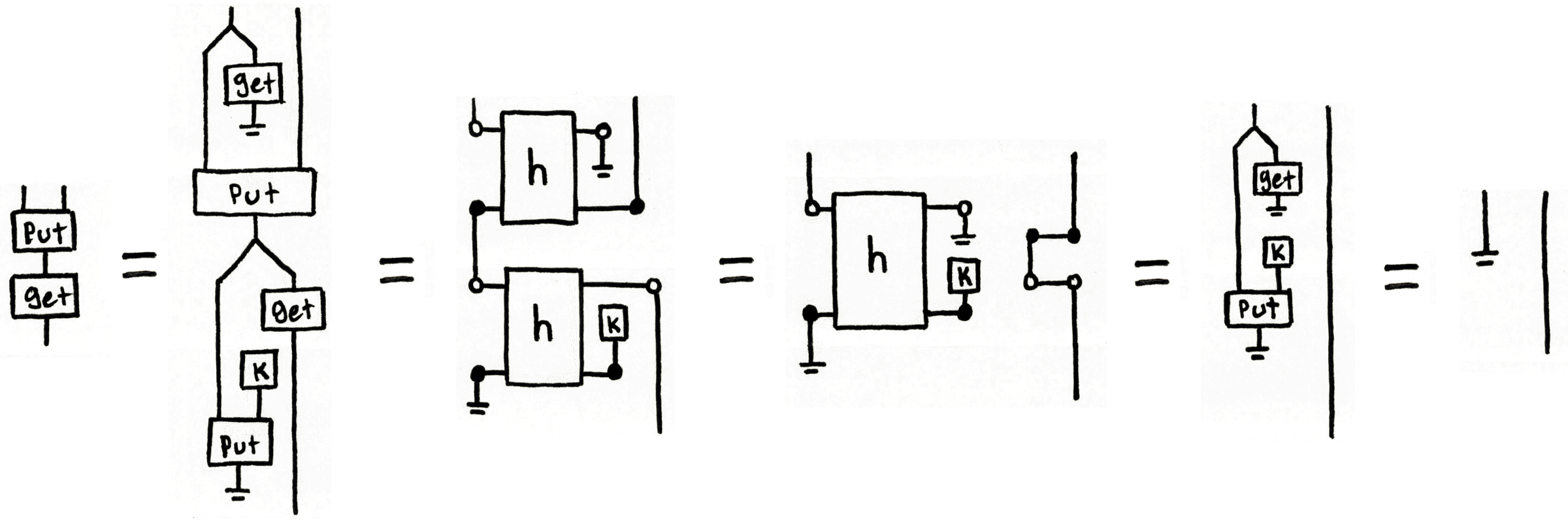}
  \]
  and the third lens law holds as in:
  \[
  \includegraphics[height=4cm,align=c]{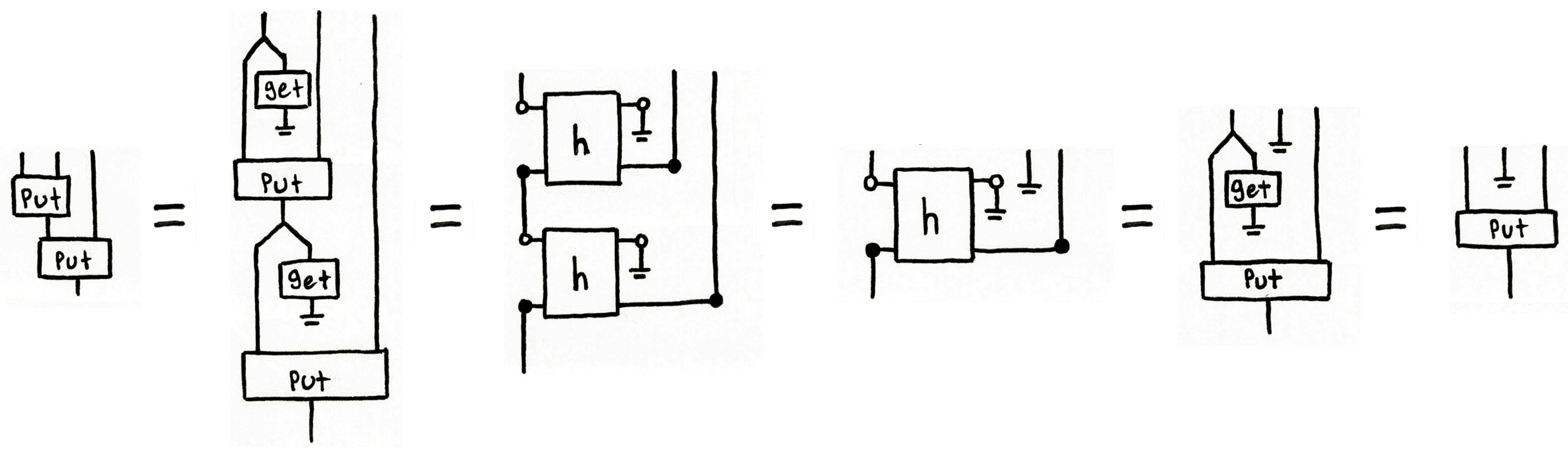}
  \]
\end{proof}

\begin{observation}[Teleological Categories]\label{obs:tele}
  $\bh\,\corner{\A}$ contains structure reminiscent of \emph{teleological categories}~\cite{Hedges2017}, which were introduced to allow well-founded diagrammatic reasoning about lenses. Analogous to the \emph{dualizable} morphisms of a teleological category are those of the form $f^\circ$, defined as below left, with duals $f^\bullet$, defined as below right:
  \begin{mathpar}
    \includegraphics[height=1.2cm,align=c]{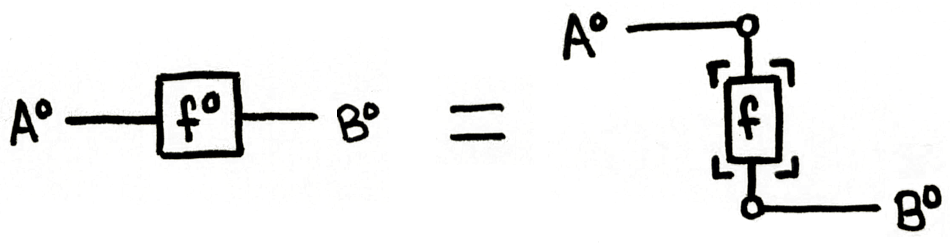}

    \includegraphics[height=1.2cm,align=c]{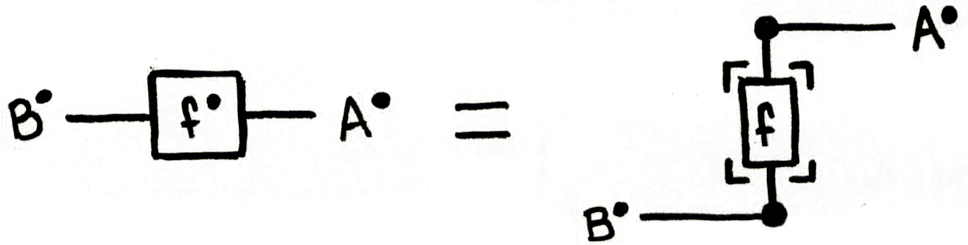}
  \end{mathpar}
  Standing in for the \emph{counits} of a teleological category we have the following cell for each $A \in \A$:
  \[
    \includegraphics[height=1.2cm,align=c]{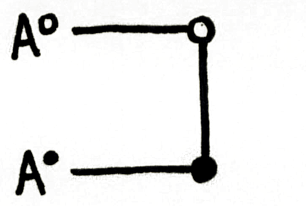}
  \]
  We then obtain an analogue of the condition that the counits be \emph{extranatural} as in:
  \[
   \includegraphics[height=1.2cm,align=c]{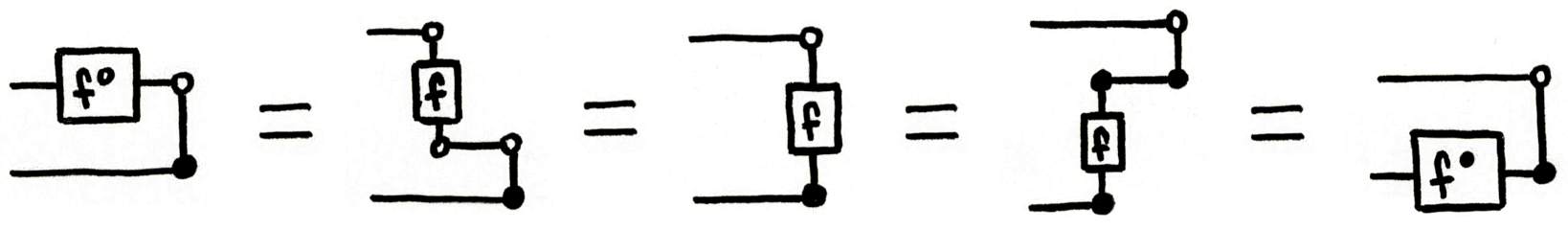}
  \]
  Notice that all arrows $A^\circ \to B^\circ$ of $\bh\,\corner{\A}$ are of the form $f^\circ$ for some $f : A \to B$ in $\A$ and that dually all arrows $B^\bullet \to A^\bullet$ are of the form $f^\bullet$, further characterising our analogue of the dualizable morphisms. 

  In light of this, we suggest that teleological categories are a shadow of the fact that $A^\circ$ is formally left adjoint to $A^\bullet$ in $\bh\,\corner{\A}$. We also point out that teleological categories do not contain enough of the relevant structure to prove Lemmas~\ref{lem:lawsimplylawful} and~\ref{lem:lawfulimplieslaws}, which require the unit of the formal adjunction between $A^\circ$ and $A^\bullet$ as well as the counit. 
\end{observation}

\section{Comb Diagrams}\label{sec:combs}

In this section we discuss comb diagrams in the free cornering. The basic idea is that we would like to have \emph{higher-order} diagrams for our monoidal categories, pictured below on the left. Supplying the appropriate first-order string diagrams to a higher-order diagram results in a first-order diagram, pictured below on the right:
\begin{mathpar}
  \includegraphics[height=2cm,align=c]{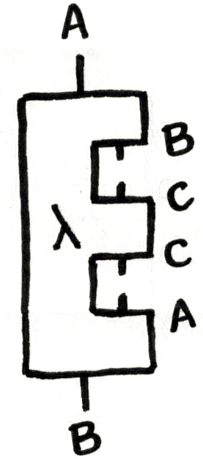}

  \includegraphics[height=2.5cm,align=c]{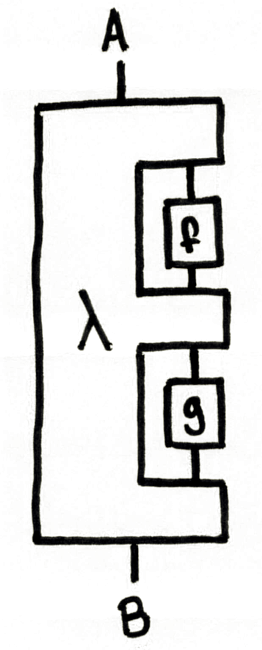}
\end{mathpar}
These higher-order diagrams have been called \emph{(right) comb diagrams} due to their appearance.

In the free cornering of a monoidal category $\A$, elements of right-$\bullet\circ$-alternating cell-sets are a good notion of right comb diagram, with the alternation depth corresponding to the number of gaps between the teeth:
\begin{mathpar}
\includegraphics[height=2cm,align=c]{figs/lambda-comb-form.png}
\hspace{0.3cm}
\leftrightsquigarrow
\hspace{0.3cm}
\includegraphics[height=1.7cm,align=c]{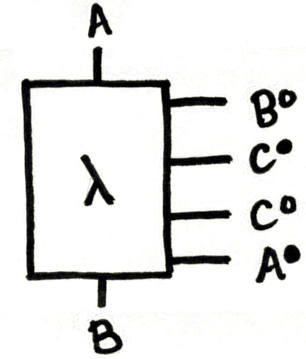}

\includegraphics[height=2.5cm,align=c]{figs/little-box-teeth-comb.png}
\hspace{0.3cm}
\leftrightsquigarrow
\hspace{0.3cm}
\includegraphics[height=2.5cm,align=c]{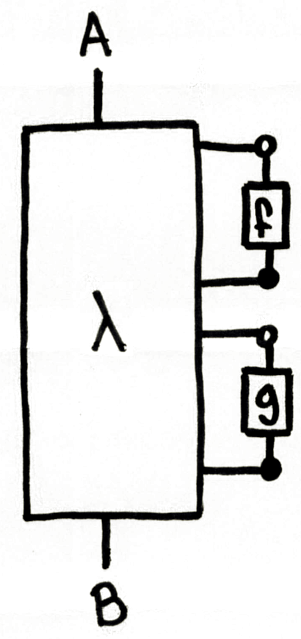}
\end{mathpar}
Lemma~\ref{lem:coendcorrespondence} tells us that this notion of comb diagram coincides with the notion of comb diagram developed by Rom\'{a}n in the more general framework of open diagrams~\cite{Roman2021}. The free cornering admits common comb diagram operations beyond inserting morphisms into the gaps. First, we may insert a comb diagram into one of the gaps to form another comb diagram:
\begin{mathpar}
\includegraphics[height=2cm,align=c]{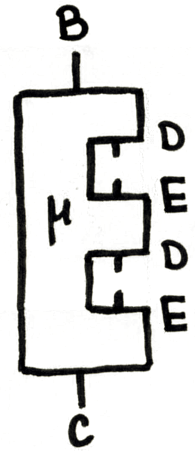}
\hspace{0.3cm}
\leftrightsquigarrow
\hspace{0.3cm}
\includegraphics[height=1.7cm,align=c]{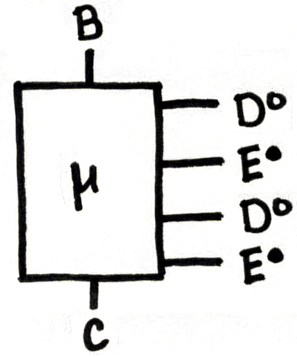}

\includegraphics[height=2.5cm,align=c]{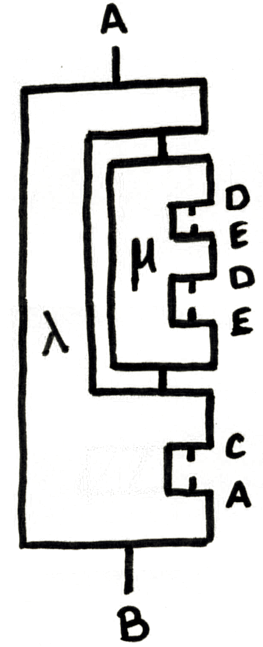}
\hspace{0.3cm}
\leftrightsquigarrow
\hspace{0.3cm}
\includegraphics[height=2.5cm,align=c]{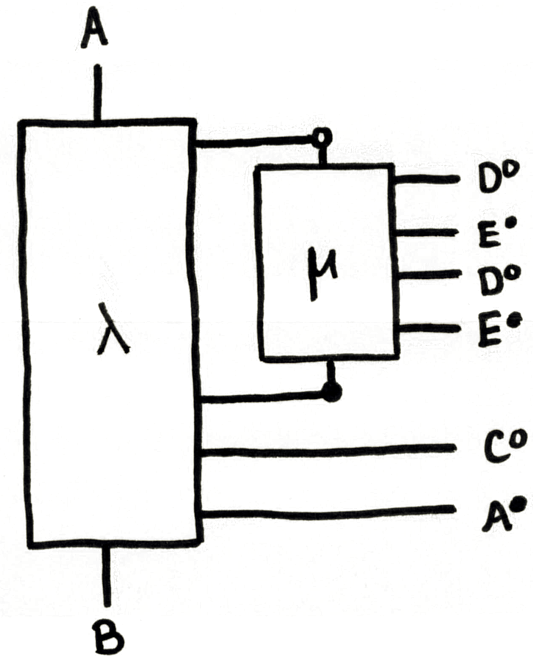}
\end{mathpar}
Next, following Remarks~\ref{rmk:alternationdual} and~\ref{rmk:opoptics} there is an dual notion of \emph{left comb diagrams} in the free cornering corresponding to the left-$\circ\bullet$-alternating cell-sets. In certain cases it makes sense to compose a right comb diagram with a left comb diagram by interleaving their teeth. The free cornering supports this as well:
\begin{mathpar}
\includegraphics[height=1.7cm,align=c]{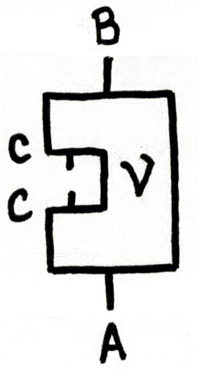}
\hspace{0.3cm}
\leftrightsquigarrow
\hspace{0.3cm}
\includegraphics[height=1.7cm,align=c]{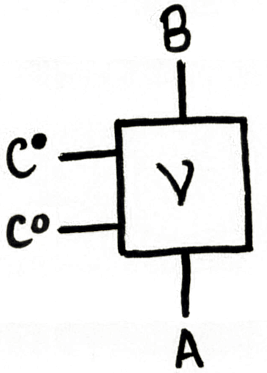}

\includegraphics[height=2.2cm,align=c]{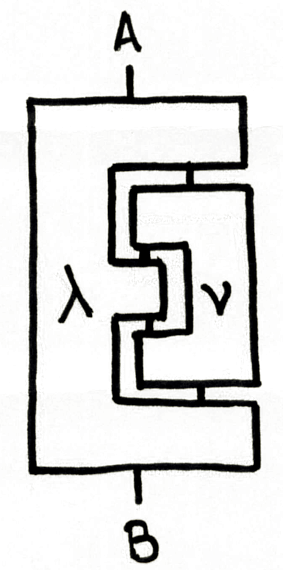}
\hspace{0.3cm}
\leftrightsquigarrow
\hspace{0.3cm}
\includegraphics[height=2.2cm,align=c]{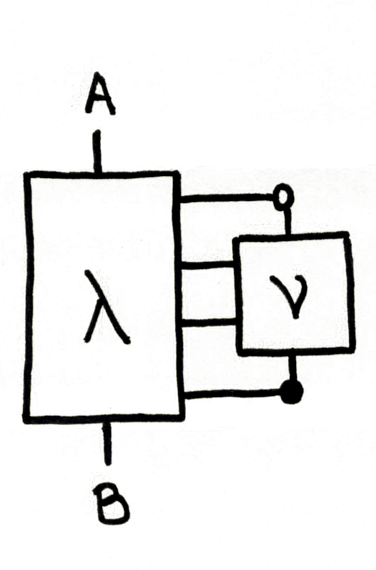}
\end{mathpar}
Thus, the free cornering is a natural setting in which to work with comb diagrams in a monoidal category.

\newpage

\bibliographystyle{eptcs}
\bibliography{citations}

\end{document}